\documentclass[10pt]{amsart}

\usepackage{tasks}
\usepackage{hyperref}
\usepackage{blindtext}
\usepackage{setspace}
\usepackage{marginnote}
\usepackage{color}
\usepackage{amsrefs}
\usepackage{mathtools}
\usepackage{graphicx}  
\usepackage{dirtytalk}
\usepackage[top=3.5cm, bottom=2.5cm, outer=2.5cm, inner=2.5cm]{geometry}
\usepackage{hyperref}
\hypersetup{linkcolor=red,
citecolor=black
}
\usepackage{amsmath}

\newtheorem{thm}{Theorem}[section]
\newtheorem{lem}[thm]{Lemma}
\newtheorem{prop}[thm]{Proposition}
\newtheorem{cor}[thm]{Corollary}

\newtheorem{lemma}[thm]{Lemma}
\theoremstyle{definition}

\theoremstyle{remark}
\newtheorem{re}[thm]{Remark}

\numberwithin{equation}{section}

\usepackage[utf8]{inputenc}
\usepackage[english]{babel}
\usepackage{paralist}
\usepackage{amsthm}
 
\theoremstyle{definition}
\newtheorem{definition}{Definition}[section]
 
\theoremstyle{remark}

 \usepackage{bm,amsmath}

\newcommand{\RN}[1]{%
  \textup{\uppercase\expandafter{\romannumeral#1}}%
}

\newcommand{\Rey}{\mathcal{R}e }

\newcommand{\grad}{\nabla}

\allowdisplaybreaks[4]
\begin{document}

\title{On  the zeroth law of turbulence for the stochastically forced Navier-Stokes equations}


\author{Yat Tin Chow}
\address{Department of Mathematics, University of California, Riverside, CA 92507, USA}
\email{yattinc@ucr.edu}

\author{Ali Pakzad}
\address{Department of Mathematics, Indiana University Bloomington, IN 47405, USA}
\email{apakzad@iu.edu}

\maketitle
\setcounter{tocdepth}{1}

\begin{abstract}
We consider  three-dimensional stochastically forced Navier–Stokes equations subjected to white-in-time (colored-in-space) forcing in the absence of boundaries.  Upper and lower bounds of the mean value of the time-averaged energy dissipation rate, $\mathbb{E} [\langle\varepsilon \rangle] $, are derived directly from the equations. First, we show that for a    weak (martingale)  solution to the stochastically forced Navier–Stokes equations, 
\[ \mathbb{E} [\langle\varepsilon \rangle]  \leq G^2 + (2+ \frac{1}{\Rey})\frac{U^3}{L},\]
where $G^2$ is the total energy rate supplied by the random force, $U$ is the root-mean-square velocity, $L$ is the longest length scale in the applied forcing function, and $\Rey$ is the Reynolds number. Under an additional assumption of energy equality, we also derive a lower bound if the energy rate given by the random force dominates the deterministic behavior of the flow in the sense that $G^2 > 2 F U$, where $F$ is the amplitude of the deterministic force. We obtain,
\[\frac{1}{3}  G^2  - \frac{1}{3} (2+ \frac{1}{\Rey})\frac{U^3}{L}   \leq  \mathbb{E} [\langle\varepsilon \rangle] \leq G^2 + (2+ \frac{1}{\Rey})\frac{U^3}{L}\,.\]
In particular, under such assumptions, we obtain \textit{the zeroth law of turbulence} in the absence of the deterministic force as,
\[\mathbb{E} [\langle\varepsilon \rangle]  = \frac{1}{2} G^2.\]
Besides,  we also obtain variance estimates of the dissipation rate for the model.
 \end{abstract}

\section{INTRODUCTION}

The stochastic Navier-Stokes equations are used as a complementary model to the deterministic one to better understand the role of small perturbations and randomness in turbulent flows.  This paper is concerned with the  stochastically forced Navier–Stokes equations on $\mathbb{T}^3$:
 \begin{equation} \label{SNSE}
\begin{split}
 d u +  ( u \cdot \nabla u -  \nu \Delta u + \nabla & p  ) \, dt=    f(x) \,  dt  +   g(t) \, d w (t ; \omega),\\
& \grad \cdot u  =0   ,
\end{split}
\end{equation}
which describe the motion of a viscous, incompressible fluid in a periodic domain.   In (\ref{SNSE}) the stochastic process $u$ and $p$ are the velocity field and the pressure respectively, and $\nu$ is the viscosity.  The  applied force is assumed to be a mean zero, white-in-time and colored-in-space Gaussian process, including a deterministic part $f$, and a stochastic part as a Wiener process on a separable Hilbert space  given by,
$$g(t) \, d w (t : \omega) := \sum\limits_{k} g_k(t) \, e_k (x) \, dW_k(t; \omega).$$
Here  $\{W_k\}$  is a countable family of independent $3$-dimensional Brownian motions over a stochastic basis $(\Omega, \mathcal{A}, \mathcal{F}, \mathbb{P})$,  and $\{g_k\}$ is a family of  given functions such that $\sum\limits_{k}  |g_k(t)|^2 < \infty$,  and $\{ e_k \}$ is a countable family of orthonormal  basis in that particular Hilbert space.   The goal of this paper is to study the effect of the noise on a key measurement of turbulent flow, namely the dissipation rate.


Adding a term representing the white noise to the basic governing equations is natural for both practical and theoretical applications.  The stochastically-forced term can be used to account for both numerical and empirical uncertainties. In particular, in the context of fluid modeling, complex phenomena related to turbulence may be modeled by stochastic perturbation, considering the fact that the onset of turbulence is often related to the randomness of background movement (see Chapter 3 of \cite{P00}).  In addition, there are many examples which support the stabilization of Navier-Stokes equations (and other PDEs) by noise (see, e.g. \cite{B01},  \cite{CLM01},  \cite{FSQD19}, \cite{K99}).    The study of the connections between Navier–Stokes equations and stochastic evolution has a long history.  This can be traced back to a work of Bensoussan and Temam \cite{BT73} in 1973.  Since then there have been a lot of studies on the stochastic Navier-Stokes equations in literature, see for example  \cite{BCPW19},  \cite{B00},  \cite{BP00}, \cite{CI08}, \cite{CI11}, \cite{MR04}, \cite{MR05}, and  \cite{WW15} and the references therein.   In the study of evolution equations of stochastic Navier-Stokes, one can consider weak
solutions of martingale type or strong solutions (see \cite{B00} and \cite{FG95} and the references therein for more details on the difference between strong and martingale solutions in this context). In this paper we consider martingale solutions (Definition \ref{0toT}),  which are weak in both the sense of PDE theory and stochastic analysis.  It is worth mentioning that martingale solutions exist and satisfy the energy inequality (\ref{EnergyIneq1}) \cite{FG95,R06}.


On the other hand, theoretical studies of turbulence usually employ a statistical description. Indeed, much of the classical turbulence theories, such as the famous Kolmogorov's conventional turbulence theory, are presented in the statistical forms (see the book by Frisch \cite{F95}).  The bulk (space and time-averaged) dissipation rate per unit mass $\langle \epsilon \rangle$ is one quantity of particular interest due to its production as a result of the turbulent cascade in the high Reynolds number vanishing viscosity limit \cite{D09}.   Kolmogorov  argued that in a turbulent fluid at large Reynolds number,  the energy dissipation rate per unit volume is essentially independent of the viscosity \cite{K41}.    Based on the concept of the energy cascade in turbulence, the rate of energy dissipation corresponds to the rate of transfer of energy from large to small scales. Hence, by a dimensional consideration, the energy dissipation rate per unit volume must take the form constant times  $\frac{U^3}{L}$,  where $U$ and $L$ are  global velocity and length scale.  Moreover, in the low Reynolds number,  the rate of energy dissipation in laminar flow scaled as  $\frac{1}{\Rey} \frac{U^3}{L}$. Therefore,  the expectation is that as $\Rey$ increases,  the flow will cross over from a laminar state  to a turbulent one with overall dissipation  independent of the viscosity \cite{D09}, \cite{F95}, \cite{P00}, i.e.,  

\begin{equation}\label{E-Scaling}
\langle \varepsilon\rangle \sim  \big(1+ \frac{1}{\Rey} \big)\frac{U^3}{L}.
\end{equation}

\subsection*{Zeroth law of turbulence} The  zeroth law of turbulence states that for fixed forcing, the rate of energy dissipation tends to a nonzero constant as the viscosity vanishes $\nu \rightarrow 0.$\\

Up to now no rigorous proof of this fact has been given.  A rigorous upper bound exhibiting this property has been found  by Doering and Constantin \cite{DC92} in 1992 for a shear boundary driven channel flow, and by Doering and Foias  \cite{DF02} in 2002 for  body-force-driven steady-state turbulence directly from the incompressible Navier-Stokes equations. These works build on Busse \cite{B70}, Howard \cite{H72}, and others.   Later, these approaches was substantially developed and extended by many authors (see, e.g., \cite{BJMT14},  \cite{DLPRSZ18},  \cite{DC96}, \cite{K97}, \cite{L16},  \cite{L02},  \cite{M94}, \cite{AP16}, \cite{AP18}, \cite{AP19},  \cite{S98},   \cite{W10},  \cite{W00},  \cite{W97},   etc.). For a good review on the recent results, including theoretical, computational and experimental, see the paper \cite{V14}.

In this paper, we aim to provide a mathematically rigorous derivation of the zeroth law of turbulence for solutions to (\ref{SNSE}) with some
hypotheses.  We first  focus on stochastically weak (martingale) solutions and  provide upper bounds on the mean value of the dissipation rate, $\mathbb{E} [\langle\varepsilon \rangle] $,  for this class of solutions.   We  then present a lower bound on  $\mathbb{E} [\langle\varepsilon \rangle] $ under an additional assumption of an energy equality.  However, a priori, one only has an energy inequality for the weak solutions, \say{measurements in all flows of real fluid in the three dimensions \textit{satisfies the energy equality}, in concert with the basic conservation laws of physics}  (page 47 of \cite{FMRT01}).  Whether a solution satisfies an energy equality or an energy inequality is related to the types of singularities that the solution may carry. It is proved in \cite{FR02}  that at every time the set of singularities is empty with probability one for a class of weak solutions of (\ref{SNSE}) which describes a fluid in a turbulent regime.    Moreover,   strong pathwise solutions  of the stochastic Navier-Stokes equations are shown to  uniquely exist  up to a maximal stopping time $\tau(\omega)$, and satisfy the energy equality up to $\tau(\omega)$  \cite{B00,strong}. We would also like to remark that  there has been also a number of papers studying conditions implying energy equality in the deterministic case,  e.g.,  \cite{DR00}, \cite{S77}, \cite{JL16}, and \cite{LS18}.

\subsection*{Organization of this paper}   The rest of this paper is organized as follows: in Section \ref{Section2}, we define the problem with its boundary and data conditions,  present definitions and the setting for the analysis. Then in Section \ref{Section3},  Theorem \ref{MainThm1},  we prove the upper bound results of these considerations.   Next in Section \ref{Section4}, Theorem \ref{MainThm2},   we prove the central result of this paper, a lower bound on the mean value of the dissipation rate,  with the extra assumption of energy balance.   We show that the hypothesis of high Reynolds number independence of the turbulent energy dissipation rate and the viscosity  holds as a lower and an upper bound for this setting.  Finally in Corollary \ref{Cor2} the exact dissipation rate, the zeroth law of turbulence, is obtained in the absence of the deterministic force.  In all these cases, a priori, the random force could pump up the  energy in the system destroying the independence of the energy dissipation of the viscosity. We prove that this is not so.  An estimate of the variance under further assumptions is also calculated in Section \ref{Section5}.  The concluding Section \ref{Section_end}  contains   some open problems in this direction.

 \section{Preliminaries}\label{Section2}

 We begin by considering some basic function spaces. Throughout this article, $D = [0, \ell]^3$ is a periodic box, i.e. we take the standand equivalent relationship $\sim$ on $\partial D$ such that $D/\sim \, \cong \mathbb{T}^3$. 
Now we denote,
$$ C^{\infty}_{\text{per}, \text{div}} (D) :=  \left\{ v \in C^{\infty}(D): \,   v  \mbox{\hspace{1pt} periodic on } D  \,, \,  \nabla \cdot v =0  \mbox{\hspace{5pt} in } D \right\}. 
$$
For any $\alpha \in \mathbb{R}$,   consider the following sobolev space,
$$ H^\alpha_{\text{per}, \text{div}}(D) := \overline{ C^{\infty}_{\text{per}, \text{div}} (D) }^{ {H}^{\alpha} (D) }\,, $$
where the standard $H^{\alpha} (D)$ norm is given by,
$$ \| \phi \|^2_{H^{\alpha} (D)} := (2 \pi)^3 \sum_{n\in \mathbb{Z}^3}(1 + \|n\|^2)^{\alpha} |\hat{\phi}(n)|^2 \,. $$
When $\alpha = 0$, we simplify the notation and denote $L^2_{\text{div}} (D) := H^0_{\text{per}, \text{div}}(D) $.  We write the standard $L^2(D)$ norm and inner product $\|\cdot\|$ and $(\cdot , \cdot)$ and drop the $H^{0} (D)$ subscript. We also denote the $L^{p}(D)$ norms by $\|\cdot\|_{p}$.  
Let us also use the same notation $(\cdot , \cdot)$ to denote the dual pairing between $H^\alpha_{\text{per}, \text{div}}(D) $ and its dual $( H^\alpha_{\text{per}, \text{div}} (D) ) '$ via the $L^2$ pivoting,
\[
H^\alpha_{\text{per}, \text{div}}(D)  \subset L^2_{\text{div}} (D) \subset ( H^\alpha_{\text{per}, \text{div}} (D) ) ' \, ,
\]
if no confusion arises.
Moreover, we have  $D(A) = H^2(D) \bigcap H^1_{\text{per}, \text{div}}(D) $, and  the   stokes operator $A$ is defined as,
 \begin{equation*}
\begin{split}
& A(u) \coloneqq   -P_L \Delta u,\\
A : D(A) &\subset L^2_{\text{div}} (D)  \rightarrow L^2_{\text{div}} (D) ,
\end{split}
\end{equation*}
where $P_L: L^2(D) \rightarrow L^2_{\text{div}} (D)$ is the Helmholtz–Leray projection.  The operator $A$ is a positive definite  with a sequence of real eigenvalues $\{\lambda_k\}_{k \in \mathbb{N}}$,  and  $\|A^{1/2} (\cdot )\|^2 \geq \lambda_1 \|\cdot\|^2$,  (see e.g. \cite{FMRT01} for more details). We also notice  $H^1_{\text{per}, \text{div}}(D) $ coincide with $ D(A^{1/2}) $, and we can endow $H^1_{\text{per}, \text{div}}(D) $ with the norm $\|A^{1/2} (\cdot )\|^2$.  

 We define the following bilinear operator $B : H^1_{\text{per}, \text{div}}(D)   \times H^1_{\text{per}, \text{div}}(D)  \rightarrow  \left( H^1_{\text{per}, \text{div}}(D)  \bigcap L^d (D) \right)' $ as, 
\begin{eqnarray*}
( B(u,v), z ) = \int_{D} z(x) \cdot (u(x) \cdot \nabla) v(x) dx
\end{eqnarray*}
for all $z \in H^1_{\text{per}, \text{div}}(D)  \bigcap L^d (D)$.   By imcompressibility, and after using integration by part,  one can show that,
\[
( B(u,v), z ) = - ( B(u,z), v ) \,.
\]
From \cite{FG95,bilinear}, $B$ can be extended continuously to 
\begin{eqnarray*}
B: L^2_{\text{div}} (D) \times L^2_{\text{div}} (D)& \rightarrow &  D(A^{-\alpha}) 
\end{eqnarray*}
for some $\alpha > 1$.

Now given $(\Omega, \mathcal{A},\mathcal{F},  \mathbb{P})$ a complete, filtered probability space equipped with the Brownian filtration satisfying usual condition, $  \mathcal{F}= \{ \mathcal{F}_t; t \in [0, T]\}$,  which is a non-decreasing family of $\sigma -$algebras, i.e.  $\mathcal{F}_t \subseteq \mathcal{F}_s$ for any $ 0 \leq t \leq s \leq T$, with completeness and right continuity.    Moreover, $\{W_k(t); t \in[0, T]\}$, $ k \in \mathbb{N}$ is a  countable family of independent $3$-dimensional Brownian motions defined on $(\Omega, \mathcal{A},\mathcal{F},  \mathbb{P})$.
The expected (mean) value for any $X$-valued $\mathcal{A}$-measurable function $\mathbb{Y}: \Omega \rightarrow X  $ can be defined as,
$$ \mathbb{E}  \left[ \mathbb{Y}\right] \coloneqq \int_{\Omega} \mathbb{Y} (\omega) \, d \mathbb{P}(\omega), $$
whenever the right-hand side exists as a Bochner integral. 
For a given Banach space $(X, \| \cdot \|_X)$, we also write the following space for $X$-valued random variable,
\begin{eqnarray*}
L^p ( \Omega, X; \mathcal{A}, \mathbb{P}) &\coloneqq& \big\{ \mathbb{X} : \Omega \rightarrow X,   \text{ $\mathcal{A}$-measurable and } \mathbb{E}  \left[ \| \mathbb{X} (\omega) \|_X^p \right] < \infty \big\} \, , \\
L^\infty ( \Omega, X; \mathcal{A}, \mathbb{P}) &\coloneqq& \big\{ \mathbb{X} : \Omega \rightarrow X,   \text{ $\mathcal{A}$-measurable and } \text{ess-sup}_{\omega} \left[ \| \mathbb{X} (\omega) \|_X \right] < \infty \big\}\,. 
\end{eqnarray*}

Moreover, given the seperable Hilbert space $L^2_{\text{div}} (D) $, we may consider a family of orthornomal basis over it, $ \{  e_k \}$, i.e. $(e_{k} , e_{l}  ) =  \delta_{kl}$.   With these notations in hand, we can now consider the following cylindrical Wiener process on $L^2_{\text{div}} (D) $ as,
$$ w (t; \omega) :=  \sum\limits_{k} e_k (x) \, W_k(t; \omega),$$
on the stochastic basis $(\Omega, \mathcal{A},\mathcal{F},  \mathbb{P})$.  
\subsection{Equations and boundary conditions}
 Consider the stochastically forced Navier-Stokes equations (\ref{SNSE}) with the periodic boundary condition, 
\begin{equation} \label{BC}
u(t, x+\ell e_j; \omega)=u(t, x;\omega)\hspace{10pt}\mbox{for any }\hspace{10pt} j=1,2,3, \,  x \in \partial D \,,  \text{  a.e. } \mathbb{P}. 
\end{equation}
We denote the stochastic forcing field $S$ as follows, 
\[
d S =  f dt + g(t) \, dw (t; \omega)  :=  f(x) dt + \sum\limits_{k} g_k(t) \, e_k (x) \, dW_k(t; \omega) \,,
\]
where  $f \in H^1_{\text{per}} (D)  $ is the deterministic part of the force,  and the stochastic noise process is given by a family of functions   $ \{  g_k \} \in C^0([0,\infty))  $ such that $\sum\limits_{k}  |g_k(t) |^2 < \infty$ for all $t \in (0,\infty)$. 
Therefore, in here, $g(t)$ is now a Hilbert–Schmidt operator given as,
\[
g(t) (\cdot ):= \sum\limits_{k} g_k(t) \, e_k (x) (e_k, \cdot ) \,.
\]
The stochastic process $S(t,\omega)$ can now be checked to be with mean $ t f(x) $ and a trace-class co-variance operator (denoted by $g^* g (t)$) defined as follows,
$$(g^* g (t) u , v) \coloneqq  \sum\limits_{k} |g_k(t)|^2 (e_k, u) (e_k, v), \text{ for all } u,v \in  L^2(D)  ,$$
$$\text{Tr} (g^* g (t)) \coloneqq  \sum\limits_{k}  |g_k(t) |^2 < \infty.$$
 We restrict our attention to smooth data $u_0(x)$ and $f(x)$ such that they are $\ell$ -periodic and divergence free. Here, $u_0$ is the initial condition such that  $\mathbb{E} [\|u_0\|^2] < \infty $. In addition, we consider mean-zero body forces and initial conditions so the velocity remains mean-zero for all $t>0$, i.e.
\begin{equation} \label{DataCond}
\begin{split}
&\hspace{5pt}\grad \cdot f= 0 \hspace{12pt}\mbox{and  }\hspace{12pt} \grad \cdot u_0=0, \\
& \int_{\Omega} \chi \,dx =0 \hspace{10pt} \mbox{for any }  \hspace{10pt}\chi= u, u_0, f, p.
\end{split}
\end{equation}

\begin{definition}\textbf{(Mean Value of the Dissipation Rate)}
 We will consider time-averaged quantity with the following notation,
$$\langle\psi\rangle \coloneqq \limsup\limits_{T\rightarrow\infty}  \, \frac{1}{T} \int_{0}^{T} \psi(t)\, dt,$$
whenever the integral exists for all $t$.   The time-averaged energy dissipation rate for  (\ref{SNSE}), which  includes dissipation due to the viscous forces,  is now a stochastic process given by,
$$\langle\varepsilon \rangle (\omega)  =  \langle\varepsilon (u (\cdot, \cdot, \omega ))\rangle   \coloneqq  \limsup\limits_{T\rightarrow\infty}  \frac{1}{|D|}\, \frac{1}{T} \int_{0}^{T}   \nu\|\nabla u(t ,\cdot, \omega)\|^2  \, dt. $$
We call   
$$\mathbb{E} [\langle\varepsilon \rangle]  \coloneqq \int_{\Omega}  \langle  \varepsilon \rangle d \mathbb{P}, $$ 
the expected value (mean value) of the dissipation rate, which will be proven to be well-defined in Corollary \ref{Cor1}.

\end{definition}

\begin{definition}\label{Scales}
With $|D|$ the volume of the flow domain, we define,

$$F \coloneqq  \langle\frac{1}{|D|  }\|f\|^2\rangle^{\frac{1}{2}}, \hspace{1.5cm}
 G  \coloneqq   \langle \frac{1}{|D|  } \sum\limits_{k} |g_k|^2 \rangle^{\frac{1}{2}}, \hspace{1.5cm}
U \coloneqq \mathbb{E}\left[ \langle\frac{1}{|D|} \|u\|^2\rangle^{\frac{1}{2}}\right],$$
and,
$$L \coloneqq \begin{cases} \min \left\{\ell , \, \frac{F}{\langle\frac{1}{|D|} \|\grad f\|^2\rangle^{\frac{1}{2}}},\, \frac{F}{\|\grad f\|_{L^{\infty}(0,T;L^{\infty}(D))}}\right\} & \text{ if } F > 0 \\ \ell & \text{ if } F = 0 \end{cases}\, .$$

\end{definition}
We will refer to $F$  as the amplitude of the deterministic force, and  $G^2$ is the \textit{total energy rate supplied by the random force}, which  has a unit of $\frac{\mbox{(velocity)}^2}{\mbox{time}}$ , we only consider the case when $G < \infty$.   With $U$ and $L$ being the large scale velocity and length,  the Reynolds number is, $$\Rey=\frac{U L}{\nu}.$$
\begin{re}
We will show in Corollary \ref{Cor1} that $\langle\frac{1}{|D|} \|u\|^2\rangle $ exists in $ L^1 ( \Omega ; \mathbb{R} ;  \mathcal{A},  \mathbb{P}) $.
\end{re}

\subsection{Martingale Solution}

Martingale solutions to (\ref{SNSE}), which are probabilistically weak analogues to the Leray-Hopf weak solutions to the deterministic
Navier–Stokes equations., were introduced first in \cite{FG95}.   Probabilistically weak means that the noise is part of the solution, along with  stochastic process $u(x, t; \omega)$ that satisfies (\ref{SNSE}) in the sense of distributions. A martingale solution can be defined as follow.

\begin{definition}  \label{0toT}\textbf{
(Martingale Solution)} A martingale solution to (\ref{SNSE}) on $[0, T]$ consists of a  stochastic basis $(\Omega, \mathcal{A},\mathcal{F},  \mathbb{P})$, with a cylindrical Wiener process $w$ over the basis,  and a progressively measurable process $u: [0,T] \times \Omega \rightarrow L^2_{\text{div}}(D)$, with $\mathbb{P}$-a.e. paths,
\begin{equation}
 u(\cdot, \cdot; \omega) \in \mathcal{C} ([0,T], D(A^{-\alpha})) \bigcap L^{\infty} ([0,T], L^2_{\text{div}}(D)  ) \bigcap L^{2} ([0,T], H^1_{\text{per}, \text{div}}(D) ),
\end{equation}
for some $\alpha < 0 $, such that, for all $t \in [0,T]$ and $\varphi \in D(A^{\alpha}) $, the following holds $\mathbb{P}$-almost surely,
 \begin{equation}\label{weak_weak}
 \begin{split}
 & \left( u(t,\cdot;\omega) , \varphi\right)
   + \int_0^t  \left( u (t,\cdot;\omega) \cdot \nabla u (t,\cdot;\omega) , \varphi \right)  dt  +  \nu \int_0^t \left( \nabla u (t,\cdot;\omega) , \nabla \varphi \right) dt   \noindent \\ 
&  = \left( u(0,\cdot;\omega) , \varphi\right) + t \left( f , \varphi \right) +  \sum\limits_{k} \int_0^t g_k(t)  \left ( e_k , \varphi \right) \, dW_k(t; \omega),
 \end{split}
 \end{equation}
and moreover $u$ satisfying the following energy inequality for every $t > s$,
  \begin{equation}\label{EnergyEq2a}
  \mathbb{E} \left[ \|u(t)\|^2 \right] + 2    \mathbb{E} \left[ \int_s^t \nu \|\nabla u(r)\|^2 dr \right] \leq    \mathbb{E} \left[ \|u(s)\|^2 \right] + \int_s^t \text{Tr} (g^* g (r)) dr  +  2  \mathbb{E} \left[ \int_s^t  (f , u(r)) dr  \right]   \,.
  \end{equation}

\end{definition}

\begin{re}\label{FutureWork}
The existence of martingale solution with Dirichlet boundary condition is shown in \cite{FG95} under very general hypotheses on the diffusion term which is independent of $t$, with the choice of filtration $\mathcal{F} = \{ \mathcal{F}_{t}; t \in [0,T] \}$ as,
\begin{equation}
\mathcal{F}_{t} = \sigma \{ u(s) ; s \leq t\} \,,
\end{equation}
the history of $u(s)$ for $s \leq t$.
Moreover, for $\beta > \max\{ \frac{3}{2} , \alpha \}$, we have that $ A^{- \beta/2} M(t, \cdot; \omega )$ is a square integrable martingale with respect to the above choice of filtration, where, 
 \begin{eqnarray} \label{martingale}
 M(t, \cdot; \omega ) := u(t,\cdot;\omega) -  u(0,\cdot;\omega) 
   + \int_0^t  u (t,\cdot;\omega) \cdot \nabla u (t,\cdot;\omega)  dt  -  \nu \int_0^t \Delta u (t,\cdot;\omega) dt  - t  f ,
\end{eqnarray}
for $\mathbb{P}$-almost surely pathwise in $C([0,t]; D(A^{- \beta/2}))$.
The existence argument in Theorem 3.1 in \cite{FG95}, with the Galerkin approximation, as well as compact embedding (Theorem 2.2 in \cite{FG95}), weak convergences, Skorohod embedding theorem and martingale representation theorem, follows identically in our case when $g$ is $t$-dependent but smooth and periodic boundary condition is employed instead of Dirichlet boundary condition (e.g. Theorem 2.2 in \cite{FG95} is valid with our choice of spaces). 

\end{re}

\begin{definition}  \label{0toinfty}
We call a progressively measurable function $ u : [0,T] \times \Omega \rightarrow L^2_{\text{div}}(D)$ w.r.t. $ \mathcal{F}$ a  martingale solution to \eqref{SNSE} on $[0,\infty)$ if $u(\cdot,x ; \omega) \mid_{[0,T]}$ is a  martingale solution to \eqref{SNSE} on $[0,T]$ for all $T >0$.
\end{definition}


\section{UPPER BOUNDS ON  DISSIPATION RATE}\label{Section3}

In this section,  we will derive the basic relationships used to establish the fundamental results, and then prove the key theorems providing the upper bounds on  $ \mathbb{E} [\langle\varepsilon \rangle]$   in terms of $U, L, G$ and $\Rey$. To a great extent,  the analysis for the stochastic NSE (\ref{SNSE}) will be considered a refinement of the approach in \cite{DF02} with careful treatment of the white noise term.

First  consider (weak) martingale solutions to (\ref{SNSE}) on $[0,\infty)$  given in Definition \ref{0toinfty}.  
Following \cite{BT73}, we consider the process $ z:= u - \Upsilon$, where  $d \Upsilon  = g(t) \, d w (t : \omega)$, with $w$ being a cylindrical Wiener process on $L^2_{\text{div}} (D) $.  We quickly notice that $z$ satisfies
\begin{equation*} 
d z=  \left( - u \cdot \nabla u + \nu \Delta u - \nabla p +  f(x) \right) dt 
\end{equation*}
which shows that $z$ is now an absolute continuous process. Moreover, following Corollary 3.1 (which comes as a direct consequence of Theorem 3.2  in \cite{BT73}), we have
\begin{equation*} 
\partial_t z \in L^2([0,T] , D( A^{-\alpha})) \text{ a.e. } \mathbb{P} \,,
\end{equation*}
and furthermore, the following energy inequality holds a.e. $\mathbb{P}$ (upon an integration by part of (3.13) in \cite{BT73} with respect to $t$ and applying the Ito's product formula), 
  \begin{equation}\label{EnergyEq2}
  \begin{split}
  \|u(T)\|^2 + 2  \int_0^T \nu \|\nabla u(t)\|^2 dt \leq  \|u_0\|^2  + \int_0^T \text{Tr} (g^* g (t)) dt &+  2 \int_0^T  (f(x), u(t)) dt \\
  &+ 2 \sum\limits_{k} \int_0^T g_k(t) (e_k, u(t)) dW_k(t).
  \end{split}
  \end{equation}
The above integral inequality can be formally short-handed  as the following stochastic differential inequation:
 \begin{equation}\label{EnergyIneq1}
 d\|u(t)\|^2 + 2 \nu \|\nabla u (t) \|^2 dt \leq \text{Tr} (g^* g (t)) dt + 2 (f(x), u(t)) dt + 2  \sum\limits_{k} g_k(t) (e_k, u(t)) dW_k(t).
  \end{equation}
 The above  inequality makes sense with the right hand side a $\mathbb{P}$-almost surely non-negative.  The above inequalities  (\ref{EnergyEq2}) and (\ref{EnergyIneq1}) are the starting points for the calculation of the upper bound. To move forward rigorously, we first   prove in Proposition \ref{KEbound} the  boundedness of the kinetic energy, and then  in  Corollary \ref{Cor1}  that $U$ given in Definition \ref{Scales}  and  $\mathbb{E} [\langle\varepsilon \rangle] $ are well-defined.  
\begin{prop}\label{KEbound} The mean value of the kinetic energy of a mean-zero martingale solution to (\ref{SNSE}) on $[0,\infty)$ is uniformly bounded in time, 
$$
 \sup_{t} \mathbb{E} [\|u(t)\|^2]\leq \ C (\mbox{data}) < \infty.
$$
\end{prop}
\begin{proof}
From Definitions \ref{0toT} and \ref{0toinfty}, a martingale solution to (\ref{SNSE}) on $[0,\infty)$ 
is in   $L^2 ( \Omega, L^{2} ([0,t], H^1_{\text{per}, \text{div}}(D) ) ; \mathcal{A}, \mathbb{P})$ for all $t > 0$. (Notice that one may obtain this conclusion either directly from \eqref{EnergyEq2a}, or from the weak lower semicontinuity of norm together with Fatou's Lemma and the inequality before (12) on P.377 of \cite{FG95}, that
\[
\mathbb{E} \left[\int_0^t\|u(s)\|^2 ds \right] \leq \liminf_{n\rightarrow \infty} \mathbb{E} \left[\int_0^t\|u_n^1(s)\|^2 ds \right] \leq C_2 <  \infty \,.
\]
Therefore, we have $\mathbb{E} [\int_0^t\|u(s)\|^2 ds ] < \infty$ for all $t >0$. 
With the given fact that $\sum\limits_{k}  |g_k(t)|^2 < \infty$ for all $t \in (0,\infty)$, together with the independence of $W_k$, orthogornality of $\{e_k\}_k $ in $L^2_{\text{div}} (D) $,  and It\^{o}'s Isometry \cite{E13}  we have,  for each $t > 0$,
\begin{eqnarray*}
\mathbb{E} \left[  \left( \sum_k \int_0^t g_k(s) (e_k , u(s)) \, dW_k(s) \right)^2 \right] &= & \mathbb{E} \left( \sum_k \int_0^t |g_k(s)|^2  |(e_k , u(s)) |^2 ds \right) \\
& \leq&    \mathbb{E} \left(  \int_0^t  \sup_k  \{ | g_k(s) |^2 \} \|u(s) \|^2 ds \right)  \\
& \leq&    \sup_{s \in [0,t] }  \left \{ \sum\limits_{k}  |g_k(t)|^2  \right\}  \, \mathbb{E} \left(  \int_0^t  \|u(s) \|^2 ds \right)  <   \infty. 
\end{eqnarray*}
%
Hence, we realize (with our simple and specific choice of $g$) that $M(t, \cdot; \omega )$ in \eqref{martingale} is square integrable martingale (without the necessity of the application of $A^{- \beta/2}$), and therefore with the standard property of It\^{o} integral  \cite{E13}, we have, 
\begin{equation}
\mathbb{E} \left[ \sum_k \int_0^T g_k(t) (e_k , u(t)) \, dW_k(t)\right] = 0. 
\end{equation}
Now taking conditional expectation of the inequality (\ref{EnergyEq2}) with respect to $\mathcal{F}_s$, and using Poincar\'{e} inequalities twice, we have, 
\begin{eqnarray*}
  \mathbb{E} [ \|u(t)\|^2 |  \mathcal{F}_s  ] & \leq & \|u(s)\|^2 + \left(  2 G^2 + \frac{ F^2}{C_{D} \nu } \right) |D| \, (t-s)  -     \mathbb{E} \left[\int_s^t \nu \| \nabla u(r)\|^2 dr |  \mathcal{F}_s  \right]  \\
& \leq & \|u(s)\|^2 + \left(  2 G^2 + \frac{ F^2}{C_{D} \nu } \right) |D| \, (t-s)  -  \mathbb{E} \left[  \frac{1}{C_{D} }  \int_s^t \nu \| u(r)\|^2 dr |  \mathcal{F}_s  \right] , 
\end{eqnarray*}
for large enough $t$ and $s$, where $C_{D}$ is a constant depending on  $D$.  Now considering $\mathbb{Y} (t) = \exp( C_{D}^{-1} \nu t  ) \| u(t)\|^2$,  and then using  It\^{o}'s formula  gives, 
\[
  \mathbb{E} [ \mathbb{Y} (t) | \mathcal{F}_s ]  \leq \mathbb{Y} (s) +  \left[ \exp( C_{D}^{-1} \nu t  )  -  \exp( C_{D}^{-1} \nu s  ) \right] \left(  2  \nu^{-1} G^2 C_{D}  + \nu^{-1} F^2 \right) |D| \,.
\]
for large enough $t$ and $s$. We hence have by submartingale convergence theorem \cite{E13} that $\mathbb{Y} (t)$ converges and,
\[
 \sup_{ 0 < t<T} \mathbb{E} [\mathbb{Y} (t) ] < C\exp( C_{\Omega}^{-1} \nu T  )  \,.
\]
Therefore, we proved the proposition.
\end{proof}

\begin{cor} \label{Cor1}
For a given martingale solution to (\ref{SNSE}) on $[0,\infty)$, we have $\langle \varepsilon \rangle$ and $  \langle\frac{1}{|D|} \|u\|^2\rangle $ both exist in $ L^1 ( \Omega , \mathbb{R} ;  \mathcal{A},  \mathbb{P}) $.
\end{cor}

\begin{proof}
We first show $  \langle\frac{1}{|D|} \|u\|^2\rangle \in L^1 ( \Omega , \mathbb{R} ;  \mathcal{A},  \mathbb{P}) $.
We notice that
\[
 0 \leq  \frac{1}{|D|}\, \frac{1}{T} \int_{0}^{T}   \nu\| u(t,\cdot, \omega)\|^2  \, dt \leq    \frac{1}{|D|} \sup_{0  \leq t }  [\|u(t, \cdot, \omega)\|^2],
\]
where the latter sits in $L^1( \Omega, \mathbb{R} ; \mathcal{F}, \mathbb{P} )$ after applying Doob's inequality on $\| u(t,\cdot, \omega)\|$:
\[
   \mathbb{E}  \left( \frac{1}{|D|} \sup_{0  \leq t}  [\|u(t, \cdot, \omega)\|^2] \right) \leq C   \left( \frac{1}{|D|} \sup_{0  \leq t } \mathbb{E} [\|u(t, \cdot, \omega)\|^2] \right) < \infty
\]
in the proof of Proposition \ref{KEbound}.
Hence we have by dominated convergence theorem that,  
\begin{eqnarray*}
 \langle\frac{1}{|D|} \|u\|^2\rangle (\omega) \coloneqq  \limsup_{T \rightarrow \infty} \frac{1}{|D|}\, \frac{1}{T} \int_{0}^{T}   \| u(t,\cdot, \omega)\|^2  \, dt,  
\end{eqnarray*}
exists in $L^1 ( \Omega , \mathbb{R} ;  \mathcal{A},  \mathbb{P})$ and 
\begin{eqnarray*}
U  =  \limsup_{T \rightarrow \infty} \mathbb{E} \left[ \frac{1}{|D|}\, \frac{1}{T} \int_{0}^{T}   \| u(t,\cdot, \omega)\|^2  \, dt  \right] <  \infty. 
\end{eqnarray*}
We next show  $ \langle \varepsilon \rangle \in L^1 ( \Omega , \mathbb{R} ;  \mathcal{A},  \mathbb{P}) $.
In fact, from \eqref{EnergyEq2} and Cauchy inequality, for large enough $T$, we have,
  \begin{eqnarray*}
   \frac{2}{T}  \int_0^T \nu \|\nabla u(t)\|^2 dt \leq 10 \sup_{0  \leq t }  \|u(t, \cdot, \omega)\|^2  +   2 |D|  G^2 +   F^2 +   \frac{2}{T}  \sum\limits_{k} \int_0^T g_k(t) (e_k, u(t,\cdot;\omega)) dW_k(t; \omega), 
  \end{eqnarray*}
 $\mathbb{P}$-a.e., where we keep absorbing constants.
 For large enough $T$, by It\^{o}'s Isometry \cite{E13} and Jensen's inequality, for any indicator function $\chi_A(\omega) $ where $A$ is $\mathcal{A}$-measurable, 
\begin{equation*}
\begin{split}
& \mathbb{E} \left(  \chi^2_{A} (\omega) \left| \frac{2}{T}  \sum\limits_{k} \int_0^T g_k(t) (e_k, u(t,\cdot;\omega)) dW_k(t,\omega) \right|^2 \right)  \leq \frac{4}{T^2}   \mathbb{E} \left( \sum_k \int_0^T |g_k(s)|^2\,  \chi_{A}^2 (\omega)    |(e_k , u(s,\cdot;\omega)) |^2 ds \right)  \\
& \leq   \frac{4}{T^2}    \mathbb{E} \left( \chi^2_{A} (\omega) \int_0^T \sup_k  \{ | g_k(s) |^2 \} \|u(s,\cdot;\omega) \|^2 ds \right) \leq   \frac{4}{T^2}  \int_0^T \sup_k  \{ | g_k(s) |^2 \}   ds \left(   \mathbb{E} \left( \chi^2_{A} (\omega)   \sup_{0  \leq t }  \|u(t, \cdot, \omega)\|^2 \right)  \right)  \\
& \leq   \frac{4}{T^2}  \int_0^T \sum_k  | g_k(s) |^2  ds \left(   \mathbb{E} \left( \chi^2_{A} (\omega)   \sup_{0  \leq t }  \|u(t, \cdot, \omega)\|^2 \right)  \right) \leq   \frac{8}{T} |D|G^2 \left(   \mathbb{E} \left( \chi^2_{A} (\omega)   \sup_{0  \leq t }  \|u(t, \cdot, \omega)\|^2 \right)  \right) \\
& \leq    \mathbb{E} \left( \chi^2_{A} (\omega)   \sup_{0  \leq t }  \|u(t, \cdot, \omega)\|^2 \right),
\end{split}
\end{equation*}
therefore  for large enough $T>0$, we have
\begin{eqnarray*}
\mathbb{E} \left(  \chi_{A} (\omega) \left( \left| \frac{2}{T}  \sum\limits_{k} \int_0^T g_k(t) (e_k, u(t,\omega)) dW_k(t,\omega) \right|^2 -   \sup_{0  \leq t }  \|u(t, \cdot, \omega)\|^2 \right ) \right) < 0 \,.
  \end{eqnarray*}
Since $A$ is an aubituary measurable set, we have, for large enough $T$, that $\mathbb{P}$-a.e.,
\begin{eqnarray*}
 \left| \frac{2}{T}  \sum\limits_{k} \int_0^T g_k(t) (e_k, u(t,\omega)) dW_k(t,\omega) \right| \leq   \sup_{0  \leq t }  \|u(t, \cdot, \omega)\|  \leq  2 +  2 \sup_{0  \leq t }  \|u(t, \cdot, \omega)\|^2 \,.
  \end{eqnarray*}
Therefore, we now have  $\mathbb{P}$-a.e. that
  \begin{eqnarray} \label{amazing}
   \frac{2}{T}  \int_0^T \nu \|\nabla u(t)\|^2 dt \leq 20 \, \sup_{0  \leq t }  \|u(t, \cdot, \omega)\|^2  +   2 |D|  G^2 +   F^2  + 1,
  \end{eqnarray}
where the latter again sits in $L^1( \Omega, \mathbb{R} ; \mathcal{F}, \mathbb{P} )$ by Doob's inequality on $\| u(t,\cdot, \omega)\|$ from Proposition \ref{KEbound}.
Hence we have by dominated convergence theorem that,
\begin{eqnarray*}
\langle \varepsilon \rangle (\omega) \coloneqq  \limsup_{T \rightarrow \infty} \frac{1}{|D|}\, \frac{1}{T} \int_{0}^{T}   \nu\|\nabla u(t,\cdot, \omega)\|^2  \, dt,  
\end{eqnarray*}
exists in $L^1 ( \Omega , \mathbb{R} ;  \mathcal{A},  \mathbb{P})$. Hence,  
$$
\mathbb{E} [ \langle \varepsilon \rangle ]   =  \limsup_{T \rightarrow \infty} \mathbb{E} \left[  \frac{1}{|D|}\, \frac{1}{T} \int_{0}^{T}   \nu\|\nabla u(t,\cdot, \omega)\|^2  \right] \, dt  <  \infty.
$$
\end{proof}
Now we have all the preliminaries ready to derive the upper bound on $\mathbb{E} [ \langle \varepsilon \rangle ]$. First  notice that,
$$  \left | \frac{ \|u(T,\cdot,\omega)\|^2 - \|u_0\|^2 }{T} \right| \leq \frac{2 }{T} \sup_{0  \leq t }  [\|u(t, \cdot, \omega)\|^2],$$
where again, the right hand side is in  $L^1( \Omega,  \mathcal{F}, \mathbb{P} )$ by Doob's inequality on $\| u(t,\cdot, \omega)\|$, and therefore, 
\begin{equation}\label{1overT}
\mathbb{E} \left[ \frac{ \|u(T,\cdot,\omega)\|^2 - \|u_0\|^2 }{T} \right] =  \mathcal{O} (\frac{1}{T}) \,.
\end{equation}
Averaging (\ref{EnergyEq2}) over $[0, T ]$, applying Cauchy–Schwarz inequality, taking expected value and  using Proposition \ref{KEbound}  yields,
\begin{equation}\label{Eq-3}
\begin{split}
&\mathbb{E} \big[\frac{1}{|D|} \frac{1}{T}  \int_0^T \nu \|\nabla u(t)\|^2 dt \big] \leq \mathcal{O} (\frac{1}{T})+ \frac{1}{2}  \frac{1}{|D|} \frac{1}{T} \int_0^T \text{Tr} (g^* g (t)) dt \\  
&+ (\frac{1}{T} \frac{1}{|D|} \int_0^T \|f\|^2 dt)^{\frac{1}{2}} \, \mathbb{E}  \big[(\frac{1}{T} \frac{1}{|D|} \int_0^T \|u(t)\|^2 dt)^{\frac{1}{2}}\big] \, ,
  \end{split}
\end{equation}
after taking into account that $\sum_k \int_0^t (g_k(s), u(s)) \, dW_k(s) $ is now itself square integrable (as show in the proof of Proposition \ref{KEbound}) and the standard property of It\^{o} integral.  Take the limit superior as $T \rightarrow \infty$, which exists by Corollary \ref{Cor1}, and use the scaled defined in  Definition (\ref{Scales}) to obtain,
\begin{equation}\label{FirstIneq}
\mathbb{E} [\langle\varepsilon \rangle] \leq \frac{1}{2}  G^2  + F U.
\end{equation}
Hence, we proved the following lemma, 
\begin{lemma} For a given mean-zero  martingale solution to (\ref{SNSE}) on $[0,\infty)$, the following inequality holds, 
\begin{equation}\label{FirstIneq_combined}
 \mathbb{E} [\langle\varepsilon \rangle ] \leq \frac{1}{2}  G^2  + F U.
\end{equation}
\end{lemma}

Next we find an upper bound on $F$ to get a bound on (\ref{FirstIneq}). Taking the inner product of (\ref{SNSE}) with $f(x)$  and integrating by parts gives, 
\begin{equation}
d(u(t), f) + (u\cdot \nabla u, f)dt + (\nu \nabla u,\nabla f) dt = (f(x), f(x)) dt +  \sum\limits_{k} g_k(t)(e_k, f) dW_k(t).
\end{equation}
Averaging the above equality over $[0,T]$
yields, 
\begin{equation}\label{Eq1}
\begin{split}
\frac{(u(T) - u_0, f(x))}{T \, |D|} + \frac{1}{|D|}\frac{1}{T} \int_0^T (u\cdot \nabla u, f) dt +  \frac{1}{|D|}\frac{1}{T} \int_0^T \nu ( \nabla u,\nabla f) dt  \\
= \frac{1}{|D|}\frac{1}{T} \int_0^T  \|f\|^2 dt +  \frac{1}{|D|}\frac{1}{T} \sum\limits_{k} \int_0^T g_k(t) (e_k , f(x)) dW_k(t).
\end{split}
\end{equation}  
Again, notice that, 
$$ \frac{| (u(T) - u_0, f(x)) | }{T \, |D|}  \leq  \frac{1 }{T} \left(  \sup_{0  \leq t }  [\|u(t, \cdot, \omega)\|^2]  + \| f \|^2 \right),$$
where the right hand side now sits in $L^2( \Omega, \mathbb{R} ;  \mathcal{F}, \mathbb{P} )$ by applying Doob's inequality on $\| u(t,\cdot, \omega)\|$, then, 
$$ \mathbb{E} \left[ \frac{(u(T) - u_0, f(x))}{T \, |D|} \right] =  \mathcal{O} (\frac{1}{T}) \rightarrow 0 \hspace{10pt} \mbox{as} \hspace{10pt} T\rightarrow 0.$$
Moreover,  from $F< \infty$ and $\sum\limits_{k}  |g_k(t)|^2 < \infty$ for all $t \in (0,\infty)$, together with the independence of $\{W_k\}$ and orthogornality of $\{e_k \}$, we have by It\^{o}'s Isometry  that,
\begin{eqnarray*}
\mathbb{E} \left[  \left( \sum_k \int_0^t g_k(s) (e_k , f ) \, dW_k(s) \right)^2 \right] \leq  \sup_{s \in [0,t) }  \left \{ \sum\limits_{k}  |g_k(t)|^2  \right\}  \, \mathbb{E} \left( t \|f \|^2 \right) < \infty \,
\end{eqnarray*}
The standard property of It\^{o} integral again implies that,
\[
 \mathbb{E} \left[ \sum_k \int_0^T g_k(t) (e_k, f) \, dW_k(t)\right] = 0 \,.
\]
Therefore, considering above identity, and taking expected value in \eqref{Eq1}. and passing to the limit superior as $T \rightarrow \infty$  gives,  
\begin{equation}\label{Eq2}
F^2 = \limsup_{T \rightarrow \infty} \mathbb{E} \big[   \frac{1}{|D|}\frac{1}{T} \int_0^T (u\cdot \nabla u, f) dt +  \frac{1}{|D|}\frac{1}{T} \int_0^T \nu ( \nabla u,\nabla f) dt \big].
\end{equation}

The rest of analysis requires to approximate the two term  on the right-hand-side in the above equality. Since $\nabla \cdot u = 0 $ we have $ (u\cdot \nabla u, f) = (\nabla \cdot (u\otimes u),   f)$ , and using integration  by part we obtain,
$$\left|\frac{1}{|D|}\frac{1}{T} \int_0^T (u\cdot \nabla u, f) dt \right| \leq \frac{1}{|D|}\frac{1}{T} \int_0^T |(u \otimes u, \nabla f)| dt \leq \|\nabla f\|_{L^{\infty}} \frac{1}{|D|}\frac{1}{T} \int_0^T  \|u\|^2 dt,$$
then using  Definition (\ref{Scales}),  the first term in (\ref{Eq2}) can be estimated with dominated convergence theorem as,
\begin{equation}\label{Eq2-a}
\left |\mathbb{E} \big[   \frac{1}{|D|}\frac{1}{T} \int_0^T (u\cdot \nabla u, f) dt \right| \leq \frac{F}{L} \,   \mathbb{E} \big[   \frac{1}{|D|}\frac{1}{T} \int_0^T  \|u\|^2 dt \big] = \frac{F}{L} U + o(1), \hspace{10pt}\mbox{as} \hspace{10pt} T\rightarrow \infty.
\end{equation}
To estimate the second term in (\ref{Eq2}), by using the Cauchy-Schwarz-Young inequality and Definition  (\ref{Scales}) we have, 
\begin{equation*}
\begin{split}
\left |\frac{1}{|D|}\frac{1}{T}\int_0^T \nu ( \nabla u,\nabla f) dt \right| &\leq (\frac{1}{|D|}\frac{1}{T} \int_0^T \nu \|\nabla u\|^2 dt)^{\frac{1}{2}} \, (\frac{1}{|D|}\frac{1}{T} \int_0^T \nu \|\nabla f\|^2 dt)^{\frac{1}{2}}\\
&\leq \sqrt{\nu} \frac{F}{L} \,  (\frac{1}{|D|}\frac{1}{T} \int_0^T \nu \|\nabla u\|^2 dt)^{\frac{1}{2}}.
\end{split}
\end{equation*}
Hence the second, viscosity, term is estimated using Jensen's inequality (P. 30 \cite{E13}) as follows,
\begin{equation}\label{Eq2-b}
\begin{split}
\left| \frac{1}{|D|}\frac{1}{T} \mathbb{E}\big[\int_0^T \nu ( \nabla u,\nabla f) dt\big] \right| &\leq  \sqrt{\nu} \frac{F}{L} \, \mathbb{E}\big[ (\frac{1}{|D|}\frac{1}{T} \int_0^T \nu \|\nabla u\|^2 dt)^{\frac{1}{2}} \big]\\
& \leq  \sqrt{\nu} \frac{F}{L} \, (\mathbb{E}\big[ \frac{1}{|D|}\frac{1}{T} \int_0^T \nu \|\nabla u\|^2 dt\big] )^{\frac{1}{2}}\\
& \leq  \sqrt{\nu} \frac{F}{L} (\mathbb{E} [\langle\varepsilon \rangle] )^{\frac{1}{2}} + o(1) , \hspace{10pt} \mbox{as} \hspace{10pt} T\rightarrow \infty.
\end{split}
\end{equation}
Inserting (\ref{Eq2-a}) and (\ref{Eq2-b}) in (\ref{Eq2}, after passing $T$ to infinity, yields,
\begin{equation*}
 F \leq \frac{U^2}{L} + \frac{\sqrt{\nu}}{L}  \, (\mathbb{E} [\langle\varepsilon \rangle] )^{\frac{1}{2}}
\end{equation*}
Insert multipliers of $U^{\frac{1}{2}}$ and $U^{-\frac{1}{2}}$
in the two terms
\begin{equation*}
 F \leq \frac{U^2}{L} + \frac{\sqrt{U\,\nu}}{L}  \, \frac{(\mathbb{E} [\langle\varepsilon \rangle] )^{\frac{1}{2}}}{\sqrt{U}} \,.
\end{equation*}
Finally  applying the  Young’s inequality on the above inequality yields to an estimate on $F$ 
\begin{equation}\label{SecondIneq1}
F \leq \frac{U^2}{L} +  \frac{1}{2} \frac{U \, \nu}{L^2}+\frac{1}{2} \frac{(\mathbb{E} [\langle\varepsilon \rangle] )}{U}.
\end{equation}
Using the  above estimate  (\ref{SecondIneq1}) for $F$ in (\ref{FirstIneq_combined}) gives,
\begin{equation*}
 \mathbb{E} [\langle\varepsilon \rangle] \leq \frac{1}{2}  G^2  +  \frac{U^3}{L} +  \frac{1}{2} \frac{U^2 \, \nu}{L^2}+\frac{1}{2} (\mathbb{E} [\langle\varepsilon \rangle] ).
\end{equation*}
Therefore we have the following  bound,  
\begin{equation*} \label{final}
\mathbb{E} [\langle\varepsilon \rangle] \leq  G^2  + 2 \frac{U^3}{L} +   \frac{U^2 \, \nu}{L^2} \leq G^2 + (2+ \frac{1}{\Rey})\frac{U^3}{L}.
\end{equation*}
We have thus estimated the  upper bound  which we summarize as,

\begin{thm}\label{MainThm1}
Let  $D = [0,\ell]^3$ denote the periodic box in $3d$, and $u(x, t; \omega)$ be a mean-zero martingale  solution of the Stochastic Navier-Stokes equations on $[0,\infty)$:
 \begin{equation} 
\begin{split}
d u=  \left( - u \cdot \nabla u + \nu \Delta u - \nabla p +  f(x) \right) dt & +  \sum\limits_{k} g_k(t) e_k(x) \, dW_k(t; \omega)\hspace{10pt} \mbox{and}\hspace{10pt} \grad \cdot u =0 \hspace{10pt} \mbox{in}\,\, D, \\
& u(x,0)=u_0(x) \hspace{10pt} \mbox{in}\,\, D,
\end{split}
\end{equation}
with boundary and data conditions given by (\ref{DataCond}) and (\ref{BC}). Then the mean value of  time-averaged energy dissipation rate per unit mass,
$$\mathbb{E} [\langle\varepsilon \rangle] := \int_{\Omega} \left[ \limsup\limits_{T\rightarrow\infty}  \frac{1}{|\Omega|}\, \frac{1}{T} \int_{0}^{T}   \nu\|\nabla u (t, \cdot, \omega )\|^2  \, dt \, \right] d\mathbb{P}, $$
satisfies,
 $$  \mathbb{E} [\langle\varepsilon \rangle]  \leq G^2 + (2+ \frac{1}{\Rey})\frac{U^3}{L},$$ 
where $U$ is the mean value of the root-mean-square (space and time averaged), $L$ is the longest length scale in the applied forcing function, and   $G^2$ is the total energy rate supplied by the random force defined in Definition \ref{Scales}.

\end{thm}

The result in Theorem \ref{MainThm1} provides an upper bound for the  mean of the  dissipation rate.  This estimate is consistent with both phenomenology (\ref{E-Scaling}) and the rate proven the Navier-Stokes equations in \cite{DF02}.
\section{EXACT  DISSIPATION RATE UNDER A FURTHER ASSUMPTION}\label{Section4}
To obtain lower bounds on $\mathbb{E} [\langle\varepsilon \rangle]$, in this section, let us make a further assumption to the martingale solution to (\ref{SNSE}) on $[0,\infty)$ as follows:\\

\noindent \textbf{Assumption (A)}  Energy equality on a  martingale solution to (\ref{SNSE}) on $[0,\infty)$:
  \begin{equation}\label{EnergyEq1} 
  \mathbb{E} \left[ \|u(T)\|^2 \right] + 2    \mathbb{E} \left[ \int_0^T \nu \|\nabla u(t)\|^2 dt \right] =   \mathbb{E} \left[ \|u_0\|^2  \right]  + \int_0^T \text{Tr} (g^* g (t)) dt  +  2  \mathbb{E} \left[ \int_0^T  (f(x), u(t)) dt  \right]   \,.
  \end{equation}

\begin{re}
Whether Assumption \textbf{(A)} holds or not heavily depends on the regularity of the  martingale solution, and the singularity type that the solution may carry. Although  the energy equality has not been mathematically proven   for turbulence models even in the deterministic case, \say{measurements in all flows of real fluid in the three dimensions \textbf{\textit{satisfies the energy equality}}, in concert with the basic conservation laws of physics}, page  57 of  \cite{FMRT01}.  There are abundant shreds of evidence that Assumption \textbf{(A)} may hold:
\begin{itemize}
\item
Different versions of locally strong/weak pathwise solutions with respect to the Brownian filtration, e.g. as defined in \cite{B00,strong}, are shown to exists up to a maximal stopping time $\tau: \Omega \rightarrow [0,\infty)$ which is $\mathcal{A}$-measurable and $\mathbb{P}$-a.e. positive \cite{strong}.  Moreover, in that case, energy equality up to $\tau(\omega)$ holds $\mathbb{P}$-a.e.  In particular, under such a regularity assumption of $u$, one will be allowed to directly application of It\^{o}'s Lemma and the momentum equation (\ref{SNSE}) of $u$ to obtain via a product rule that for all $t > 0$, that $\mathbb{P}$-a.e. and $t < \tau (\omega)$,
\[
 d\|u(t)\|^2 + 2 \nu \|\nabla u (t) \|^2 dt =  \text{Tr} (g^* g (t)) dt + 2 (f(x), u(t)) dt + 2  \sum\limits_{k} g_k(t) (e_k, u(t)) dW_k(t) \,.
\]
Upon a further checking of the stochastic part being square integrable, standard property of It\^{o} integral then quickly implies for all $t >0$,
\begin{equation*}
 \begin{split}
\mathbb{E} \left[ \|u( t \wedge \tau(\omega) )\|^2 \right] + 2    \mathbb{E} \left[ \int_0^{t \wedge \tau(\omega) } \nu \|\nabla u(s)\|^2 ds \right] & =   \mathbb{E} \left[ \|u_0\|^2  \right]  + \int_0^{ t \wedge \tau(\omega) } \text{Tr} (g^* g (s)) ds\\ 
  & +  2  \mathbb{E} \left[ \int_0^{ t \wedge \tau(\omega) }  (f(x), u(s)) ds  \right].
    \end{split}
\end{equation*}
where $ t \wedge \tau(\omega) :=  \min\{t,\tau(\omega)\}$.  It is also shown that when $d =2$, one can choose $\tau (\omega) = \infty$  $\mathbb{P}$-a.e.

\item It is shown in  \cite{FR02} that at every time the set of singularities of  a class of the martingale solutions of (\ref{SNSE})  is empty with probability one. This fact means that  at every time it is not possible to see the singularities (and possibly blow-up):  only a negligible set of paths may have singularities at a fixed time. 

\end{itemize}
\end{re}

In this subsection, we always assume that Assumption \textbf{(A)} holds.
Now, similar to the way we obtain \eqref{Eq-3}, from \eqref{EnergyEq2}, combining the equality (\ref{EnergyEq1})  together with the estimate in (\ref{1overT}) we have,
\begin{equation}\label{Eq-3aa}
\begin{split}
\mathbb{E} \big[\frac{1}{|D|} \frac{1}{T}  \int_0^T \nu \|\nabla u(t)\|^2 dt \big] = \mathcal{O} (\frac{1}{T})+ \frac{1}{2}  \frac{1}{|D|} \frac{1}{T} \int_0^T \text{Tr} (g^* g (t)) dt + \mathbb{E}  \big[(\frac{1}{T} \frac{1}{|D|} \int_0^T (f, u(t) ) dt)\big] \, ,
  \end{split}
\end{equation}
Therefore, rearranging the terms and by Cauchy–Schwarz inequality,
\begin{equation*}
\begin{split}
\frac{1}{2}  \frac{1}{|D|} \frac{1}{T} & \int_0^T  \text{Tr} (g^* g (t)) dt = \mathcal{O} (\frac{1}{T}) +  \mathbb{E} \big[\frac{1}{|D|} \frac{1}{T}  \int_0^T \nu \|\nabla u(t)\|^2 dt \big] + \mathbb{E}  \big[(\frac{1}{T} \frac{1}{|D|} \int_0^T ( - f, u(t) ) dt)\big]   \\
& \leq \mathcal{O} (\frac{1}{T}) +  \mathbb{E} \big[\frac{1}{|D|} \frac{1}{T}  \int_0^T \nu \|\nabla u(t)\|^2 dt \big] 
+ (\frac{1}{T} \frac{1}{|D|} \int_0^T \|f\|^2 dt)^{\frac{1}{2}} \, \mathbb{E}  \big[(\frac{1}{T} \frac{1}{|D|} \int_0^T \|u(t)\|^2 dt)^{\frac{1}{2}}\big].
\end{split}
\end{equation*}
Therefore passing to the limit superior as $T \rightarrow \infty$, and by dominated convergence theorem again, we have,
\begin{equation}\label{SecondIneq}
\frac{1}{2}  G^2 \leq  \mathbb{E} [\langle\varepsilon \rangle] + F U.
\end{equation}
With this inequality as above, we are motivated to define the following.

\begin{definition}
The stochastically forced  NSE  \eqref{SNSE} is stochastically dominated if $  G^2 > 2 F U.$
\end{definition}
This  means the stochastic term given by the random force  $ g(t) \, d w (t;\omega) = \sum\limits_{k} g_k(t; x) \, dW_k(t; \omega)$ dominates the behaviour of the solutions, and therefore, if Assumption \textbf{(A)} holds,
$$
 0 < \frac{1}{2}  G^2  - F U \leq  \mathbb{E} [\langle\varepsilon \rangle].
$$
Moreover, applying the upper bounds (\ref{SecondIneq}) on $F$, it is then straightforward to see that

 $$\frac{1}{3}  G^2  -  (\frac{2}{3}+ \frac{1}{3}\frac{1}{\Rey})\frac{U^3}{L}  \leq  \mathbb{E} [\langle\varepsilon \rangle].$$ 
 
Considering Theorem \ref{MainThm1} and the above lower bound, we obtain the following.

\begin{thm}\label{MainThm2}
Consider assumptions in Theorem \ref{MainThm1}. In addition suppose Assumption \textbf{(A)} and that
$$ G^2 > 2 FU,$$
where $U$, $F$  and $G$ are defined in Definition \ref{Scales}. Then  $\mathbb{E} [\langle\varepsilon \rangle]$ satisfies
\begin{equation} \label{LetF=0}
\frac{1}{2}  G^2 - F U \leq  \mathbb{E} [\langle\varepsilon \rangle] \leq \frac{1}{2}  G^2 + F U \,,
\end{equation}
and moreover,
 $$\frac{1}{3}  G^2  - \frac{1}{3} (2+ \frac{1}{\Rey})\frac{U^3}{L}  \leq  \mathbb{E} [\langle\varepsilon \rangle]  \leq G^2 + (2+ \frac{1}{\Rey})\frac{U^3}{L}.$$

\end{thm}

The result in Theorem (\ref{MainThm2})  is remarkable, showing a stochastically dominated stochastically forced NSE has its dissipation rate precisely  behaving as a function of the Reynolds number, (\ref{E-Scaling}), as it is discussed in introduction. 

\begin{cor}\label{Cor2}
Considering the conditions of Theorem \ref{MainThm2}, and   in absence of deterministic force, i.e. $f=0$, we obtain the exact dissipation rate as,

$$\mathbb{E} [\langle\varepsilon \rangle]  = \frac{1}{2} G^2.$$
\end{cor}
\begin{proof}
Let $F = 0$ in (\ref{LetF=0}).
\end{proof}

\section{Variance Estimate of Energy Dissipation}\label{Section5}

We now estimate the variance, 
\[
\mathbb{V}\text{ar} (\langle \varepsilon \rangle) \coloneqq \mathbb{E} (\langle \varepsilon \rangle ^2 ) - [ \mathbb{E} (\langle \varepsilon \rangle ) ]^2
\]
if the right handside is well-defined.  In order to have the right hand side well-defined and to establish a bound, let us consider further assumptions,

\noindent \textbf{Assumption (B)}  
An energy bound of  martingale solution to (\ref{SNSE}):
  \begin{equation}\label{Energy_ineq_haha} 
 \mathbb{E} \left( \int_0^T \|u(t)\|^6 dt \right) < C(T) < \infty \,.
  \end{equation}
for all $T \in [0,\infty)$. 

\noindent \textbf{Assumption (C)}  
A bound on the noise term $dG := g(t) \, d w (t;\omega)$:
  \begin{equation}\label{Energy_ineq_G} 
\left\langle \left ( \sum\limits_{k} |g_k|^2 \right)^2 \right \rangle \coloneqq K_G^4  < \infty. 
  \end{equation}

\noindent
With the above two assumptions,
applying \eqref{EnergyIneq1} and It\^{o}'s Lemma \cite{E13} to $ \left(  \| u(t)\|^2 \right)^2$, we obtain
 \begin{equation*}
 \begin{split}
  d\|u(t)\|^4 \leq   &\left(  3 \|u(t)\|^2 \text{Tr} (g^* g (t))    +   4 \|u(t)\|^2  (f(x), u(t)) -   4   \nu \|u(t)\|^2 \|\nabla u (t) \|^2  \right) dt\\
   +  &4  \sum\limits_{k} \|u(t)\|^2 g_k(t) (e_k , u(t)) dW_k(t).
 \end{split}
  \end{equation*}
where we used $ \| [g (t)] \, u(t)\|^2  \leq  \text{Tr} (g^* g (t)) \| u(t)\|^2$.  The above is equivalent to, for all $T>0$,
 \begin{equation} \label{EnergyEq2new}
  \begin{split}
  \|u(T)\|^4 + 4  \int_0^T \nu  \|u(t)\|^2 \|\nabla u (t) \|^2  dt \leq   & \|u_0\|^4  + 3 \int_0^T  \|u(t)\|^2   \text{Tr} (g^* g (t)) dt +  4 \int_0^T \|u(t)\|^2  (f(x), u(t)) dt \\
  +& \,  4 \sum\limits_{k} \int_0^T  \|u(t)\|^2 \, g_k(t) (e_k, u(t)) dW_k(t).
  \end{split}
  \end{equation}

\begin{prop}\label{KEsqbound}
Suppose Assumptions \textbf{(B)} and \textbf{(C)} hold. For a given  martingale solution to (\ref{SNSE}) on $[0,\infty)$, we have
$$
 \sup_{t} \mathbb{E} [\|u(t)\|^4]\leq \ C (\mbox{data}) < \infty.
$$
\end{prop}
\begin{proof}
Now with $\mathbb{E} [\int_0^t\|u(s)\|^6 ds ] <  C(t) < \infty$ for all $t >0$ and $\sum\limits_{k}  |g_k(t)|^2 < \infty$ for all $t \in (0,\infty)$, together with the independence of $W_k$ and orthogornality of $\{e_k \}_k$, we have by It\^{o}'s Isometry \cite{E13} that for each $t >0$,

\begin{equation*}
\begin{split}
& \mathbb{E} \left[  \left( \sum_k \int_0^t \|u(t)\|^2  g_k(s) (e_k, u(s)) \, dW_k(s) \right)^2 \right]   =   \mathbb{E} \left( \sum_k \int_0^t \| u(s) \|^4 |g_k(s)|^2 |(e_k, u(s)) |^2 ds \right) \\
\leq & \,    \mathbb{E} \left(  \int_0^t  \sup_k  \{ | g_k(s) |^2 \}  \|u(s)\|^4  \sum_k \left|\left(e_k, u(s)\right) \right|^2 ds \right)    \leq    \mathbb{E} \left(  \int_0^t  \sup_k  \{ | g_k(s) |^2 \} \|u(s) \|^6 ds \right)  \\ 
 \leq  & \,  \sup_{s \in [0,t] }  \left \{ \sum\limits_{k}  |g_k(t) |^2  \right\}  \, \mathbb{E} \left(  \int_0^t  \|u(s) \|^6 ds \right) < \infty.\\
\end{split}
\end{equation*}

Therefore  by the standard property of It\^{o} integral  \cite{E13}, 
\[
 \mathbb{E} \left[ \sum_k \int_0^T \| u(t)\|^2 (g_k(t), u(t)) \, dW_k(t)\right] = 0 \,.
\]
With this, from (\ref{EnergyEq2new}), and via Young's and inequalities and Poincare inequalities twice, we have for large enough $t,s$
\begin{eqnarray*}
  \mathbb{E} [ \|u(t)\|^4 | \mathcal{F}_s ] & \leq & \|u(s)\|^4 + K_{D,\nu} \left(   K_G^4 +  F^4  \right) \, (t-s)  -   \frac{2}{C_{D} }  \mathbb{E} \left[ \int_s^t \nu \| u(r)\|^4 dr  | \mathcal{F}_s \right] 
\end{eqnarray*}
where $C_{D}$ is the Poincare constant of $D$ and $K_{D,\nu}$ depends on $D$ and $\mu$.  Again, via It\^{o}'s formula   \cite{E13} and submartingale convergence theorem \cite{E13}, we have
$$
 \sup_{t} \mathbb{E} [\|u(t)\|^4]\leq \ C (\mbox{data}) < \infty.
$$
\end{proof}

\begin{cor} \label{happy}
Suppose Assumptions \textbf{(B)} and \textbf{(C)} hold. For a given  martingale solution to (\ref{SNSE}) on $[0,\infty)$, we have $\langle \varepsilon \rangle^2  $ exists in $ L^1 ( \Omega ; \mathbb{R} ;  \mathcal{A},  \mathbb{P}) $.
\end{cor}

\begin{proof}
From \eqref{amazing}, we have
\begin{equation*}
\begin{split}
\left(  \frac{1}{|D|}\, \frac{1}{T} \int_{0}^{T}   \nu\|\nabla u(t,\cdot, \omega)\|^2 \right)^2
 \leq \frac{1}{|D|^2}  \Bigg( 100 \, \sup_{0  \leq t }  \|u(t, \cdot, \omega)\|^4    + 20   ( |D|  G^2 +   F^2  + 1 ) &  \sup_{0  \leq t }  \|u(t, \cdot, \omega)\|^2  \\
& + ( |D|  G^2 +   F^2  + 1 )^2 \Bigg) \,.
 \end{split}
\end{equation*}
Now, since both $\sup \limits_{t \geq 0 }  \|u(t, \cdot, \omega)\|^2$ and $\sup \limits_{ t \geq 0 }  \|u(t, \cdot, \omega)\|^4$ are in $L^1( \Omega,  \mathcal{F}, \mathbb{P} )$ after applying Doob's inequality on $\| u(t,\cdot, \omega)\|$ for $l = 1,2$:
\[
   \mathbb{E}  \left( \frac{1}{|D|} \sup_{0  \leq t}  [\|u(t, \cdot, \omega)\|^{2l}] \right) \leq C_l   \left( \frac{1}{|D|} \sup_{0  \leq t } \mathbb{E} [\|u(t, \cdot, \omega)\|^{2l}] \right) < \infty,
\]
with the last inequality coming from Propositions \ref{KEbound} and \ref{KEsqbound}.   We have, by dominated convergence theorem,
\begin{eqnarray*}
\langle \varepsilon \rangle^2 (\omega) \coloneqq  \limsup_{T \rightarrow \infty}  \left(  \frac{1}{|D|}\, \frac{1}{T} \int_{0}^{T}   \nu\|\nabla u(t,\cdot, \omega)\|^2 \right)^2  \,,
\end{eqnarray*}
exists in $L^1 ( \Omega ; \mathbb{R} ;  \mathcal{A},  \mathbb{P})$ and, 
$$
\mathbb{E} [ \langle \varepsilon \rangle^2 ]   =  \limsup_{T \rightarrow \infty} \mathbb{E} \left[ \left(  \frac{1}{|D|}\, \frac{1}{T} \int_{0}^{T}   \nu\|\nabla u(t,\cdot, \omega)\|^2 \right)^2  \right] \, dt  <  \infty. 
$$
\end{proof}
Now, squaring \eqref{EnergyEq2}, we have
  \begin{equation*}
    \begin{split}
   \left(  \frac{1}{T} \int_0^T \varepsilon dt \right)^2  \leq   \bigg( \frac{  \|u_0\|^2 -  \|u(T)\|^2  }{2 T |D|}   +   \frac{1}{|D| T } \int_0^T \text{Tr} (g^* g (t)) dt  & +  \frac{1}{|D| T }   \int_0^T  (f(x), u(t)) dt \\
&    +   \frac{1}{|D| T }  \sum\limits_{k} \int_0^T g_k(t) (e_k, u(t)) dW_k(t) \bigg)^2.
  \end{split}
  \end{equation*}
Therefore Cauchy-Schwarz inequality gives, for any $a,b,c,d > 0$,
    \begin{equation*}
  \begin{split}
 \mathbb{E} \left[ \left(  \frac{1}{T} \int_0^T \varepsilon dt \right)^2 \right]   \leq & (a^{-2} + b^{-2} + c^{-2} + d^{-2}) \bigg( a^2 \mathbb{E} \left[ \left( \frac{  \|u_0\|^2 -  \|u(T)\|^2  }{2 T |D|}  \right)^2 \right]  +  b^2  \left( \frac{1}{|D| T } \int_0^T \text{Tr} (g^* g (t)) dt \right)^2 \\
  +  & c^2 \mathbb{E} \left[ \left(  \frac{1}{|D| T }   \int_0^T  (f(x), u(t)) dt \right)^2 \right]  + d^2 \mathbb{E} \left[ \left(  \frac{1}{|D| T }  \sum\limits_{k} \int_0^T g_k(t) (e_k , u(t)) dW_k(t) \right)^2 \right] \bigg)  \\
\coloneqq  & (a^{-2} + b^{-2} + c^{-2} + d^{-2}) \left( a^2(I) + b^2(II) + c^2(III) + d^2 (IV) \right) \,.
  \end{split}
  \end{equation*}
Now,
\begin{eqnarray*}
\frac{ | \|u_0\|^2 -  \|u(T)\|^2 |^2  }{4 T^2 |D|^2} \leq \frac{1}{2T^2 |D|^2} \sup_{0  \leq t }  [\|u(t, \cdot, \omega)\|^4],
\end{eqnarray*}
where again the right hand side of the inequality is in  $L^1( \Omega, \mathbb{R} ; \mathcal{F}, \mathbb{P} )$ by Doob's inequality.  Therefore we have,
$$
(I) \leq \mathcal{O} (\frac{1}{T^2}) \,.
$$
Then by definition, 
$$
(II) = G^4 \,.
$$
Next,using Jensen's inequality we have, 
\begin{eqnarray*}
(III)  \leq  F^2 \mathbb{E} \left[  \left(  \frac{1}{ T |D| }   \int_0^T   \| u(t)\|  dt \right)^2 \right] \leq  F^2 U^2 \,.
\end{eqnarray*}
Lastly, by It\^{o}'s Isometry \cite{E13}, together with $\mathbb{E} [\int_0^t\|u(s)\|^2 ds ] < \infty$ for all $t >0$ and $\sum\limits_{k}  |g_k(t)|^2 < \infty$ for all $t \in (0,\infty)$, independence of $W_k$ and orthogonality of $\{e_k\}_k$, we get for $T$ large enough,

\begin{equation*}
\begin{split}
(IV)  & =  \mathbb{E} \left[ \frac{1}{|D|^2 T^2 }  \sum\limits_{k} \int_0^T | g_k(t) (e_k, u(t))  |^2 dt \right]  \leq
  \mathbb{E} \left[ \frac{1}{|D|^2 T^2 }   \int_0^T \sup_{k} | g_k(t)|^2   \sum\limits_{k}  \left| \left(e_k , u(t) \right)  \right|^2 dt \right] \\
  & \leq  \mathbb{E} \left[ \frac{1}{|D|^2 T^2 }   \int_0^T \sup_{k} | g_k(t) |^2   \| u(t)\|^2 dt  \right]  \leq  \mathbb{E} \left[ \frac{1}{|D|^2 T^2 }   \int_0^T \left(  \sum\limits_{k}  | g_k(t)|^2 \right)   \| u(t)\|^2 dt  \right] \\
  & \leq  \mathbb{E} \left[ \frac{\sup_t \| u(t)\|^2}{|D|^2 T^2 }   \int_0^T \left(  \sum\limits_{k}  | g_k(t)|^2 \right)   dt  \right] \leq  \frac{2 G^2}{|D|^2 T }  \,  \mathbb{E} \left[ \sup_t \| u(t)\|^2 \right] \\
  &   \leq \frac{4 G^2}{|D|^2 T }   \,  \sup_t \mathbb{E} \left[  \| u(t)\|^2 \right]  = \mathcal{O} (1/T),
\end{split}
\end{equation*}
where the last inequality comes from Doob's inequality once again.  Now combining the above four estimates, and passing to the limit superior as $T \rightarrow \infty$, together with Corollary \ref{happy}, we have, taking also infrimum over $a,b,c,d > 0$,
    \begin{equation*}
  \begin{split}
 \mathbb{E} \left[ \langle \varepsilon \rangle^2 \right]   \leq \inf_{a,b,c,d > 0} (a^{-2} + b^{-2}+ c^{-2} + d^{-2})( b^2 G^4 + c^2 F^2 U^2 ) \leq 2 G^4 + 2 F^2 U^2 .
  \end{split}
  \end{equation*}

Therefore we obtain the following estimate, combining with Theorem \ref{MainThm1} and Theorem \ref{MainThm2}.
\begin{lem} \label{happy}
Suppose Assumptions \textbf{(B)} and \textbf{(C)} hold. For a given  martingale solution to (\ref{SNSE}) on $[0,\infty)$, we have
$$
 \mathbb{E} \left[ \langle \varepsilon \rangle^2 \right] -  \left(\mathbb{E} [\langle\varepsilon \rangle] \right)^2  \leq  2 G^4 + 2 F^2 U^2 - \left( \max\left\{ \frac{1}{2} G^2 - F U, 0\right\} \right)^2.
$$
\end{lem}

Now we combine the above with \eqref{SecondIneq} and Theorem \ref{MainThm1} to have
\begin{eqnarray*}
 \mathbb{V}\text{ar} (\langle \varepsilon \rangle)  & \leq & 2 G^4 + 2 \left( \frac{U^3}{L} +  \frac{1}{2} \frac{U^2 \, \nu}{L^2}+\frac{1}{2} \mathbb{E} [\langle\varepsilon \rangle] ) \right)^2 -  \left(\mathbb{E} [\langle\varepsilon \rangle] \right)^2 \\
 & \leq & 2 G^4 + \frac{1}{2} \left(  (2 + \frac{1}{\Rey} )  \frac{U^3}{L}    + \mathbb{E} [\langle\varepsilon \rangle] ) \right)^2 -  \left(\mathbb{E} [\langle\varepsilon \rangle] \right)^2 \\
& \leq & 2 G^4 +  \frac{1}{2}   (2 + \frac{1}{\Rey} )^2  \frac{U^6}{L^2}     +  (2 + \frac{1}{\Rey} )  \frac{U^3}{L}  \, \mathbb{E} [\langle\varepsilon \rangle]  - \frac{1}{2} \left(\mathbb{E} [\langle\varepsilon \rangle] \right)^2  \\
& \leq & 2 G^4 +  (2 + \frac{1}{\Rey} )  G^2 \frac{U^3}{L}  +  \frac{3}{2}  (2 + \frac{1}{\Rey} )^2 \frac{U^6}{L^2} 
\end{eqnarray*}

\begin{thm} \label{Theorem_variance}

Consider Theorem \ref{MainThm1}.  Suppose further that Assumptions \textbf{(B)} and \textbf{(C)} hold. For a given  martingale solution to (\ref{SNSE}) on $[0,\infty)$, we have the  variance of the time averaged energy dissipation rate,
\[
\mathbb{V}\text{ar} (\langle \varepsilon \rangle) \coloneqq \mathbb{E} (\langle \varepsilon \rangle ^2 ) - [ \mathbb{E} (\langle \varepsilon \rangle ) ]^2
\]
well-defined and satisfying the following bounds:
$$
 \mathbb{V}\text{ar} (\langle \varepsilon \rangle)  \leq  2 G^4 + 2 F^2 U^2 - \left( \max\left\{ \frac{1}{2} G^2 - F U, 0\right\} \right)^2,
$$
and,
$$
 \mathbb{V}\text{ar} (\langle \varepsilon \rangle)  \leq 2 G^4 +  (2 + \frac{1}{\Rey} )  G^2 \frac{U^3}{L}  +  \frac{3}{2}  (2 + \frac{1}{\Rey} )^2 \frac{U^6}{L^2}   \,.
$$
\end{thm}
\begin{re}
 If $F = 0$, then we have $ \mathbb{V}\text{ar} (\langle \varepsilon \rangle)  \leq \frac{7}{4} G^4 $.
\end{re}

\section{Discussion}\label{Section_end}

In this paper we  proved the zeroth law of turbulence for three dimensional stochastically forced Navier-Stokes equation in the absence of the deterministic force assuming the energy balance. 

 With Remark (\ref{FutureWork})  in mind, we realize the conclusion of our work can be extended to a more general diffusion term including assumptions in \cite{FG95} that allows  a state dependent noise term  $g(t,u) \, dw(t,\omega)$ satisfying more general hypotheses.  However, for the sake of simplicity, we shall postpone the investigation along this direction to a future work.

The list of open problem might include the followings,
\begin{itemize}
\item Extension of the estimates of  $\mathbb{E} [\langle\varepsilon\rangle] $ to the channel flow case is still open and  would give insight into near wall behavior. 
\item Resolution is one basic factor which affect the accuracy of fluid simulation.  Turbulence models are introduced to account for sub-mesh scale effects, but there is always an error originated from modeling while   this error can be addressed to some extent by the introduction of noise.  The accuracy of any stochastic turbulence model, e.g., eddy viscosity model, can be studied by calibrating its expected energy dissipation rate.
\end{itemize}


\begin{thebibliography}{9}




\bib{B01}{book}{
   author={Barbu, V.},
 
   title={Stabilization of Navier–Stokes Flows},

   publisher={Springer, London}
   date={2001},
 
 
}


\bib{BCPW19}{article}{
   author={Bedrossian, J.},
      author={Coti Zelati, M.},
       author={ Punshon-Smith, S.},
        author={Weber, F.},

 TITLE = {A sufficient condition for the {K}olmogorov 4/5 law for
              stationary martingale solutions to the 3{D} {N}avier-{S}tokes
              equations},
   JOURNAL = {Comm. Math. Phys.},
  FJOURNAL = {Communications in Mathematical Physics},
    VOLUME = {367},
      YEAR = {2019},
    NUMBER = {3},
     PAGES = {1045--1075},
  
}




\bib{Batchelor }{article}{
   author={Bedrossian, J.},
      author={Blumenthal,  A.},
       author={Punshon-Smith, S.},
 TITLE = {The Batchelor spectrum of passive scalar turbulence in stochastic fluid mechanics},
   JOURNAL = {Preprint},
  FJOURNAL = {ArXiv},
     PAGES = {arXiv:1911.11014},
}




 \bib{BT73}{article}{
   author={Bensoussan, A.},
      author={Temam, R.},

   title={ Equatios stochastique du type Navier-Stokes},
   journal={J. Funct. Anal.},
   volume={13},
   date={1973},

   pages={195–222},
  
}



 \bib{BJMT14}{article}{
   author={Biswas, A. },
      author={Jolly, M. S. },
       author={Martine, V. R.},
         author={Titi, E. S. },
   title={Dissipation Length Scale Estimates for Turbulent Flows: A Wiener
Algebra Approach},
   journal={Journal of Nonlinear Science},
   volume={24},
   date={2014},

   pages={441-471},
  
}




 \bib{B00}{article}{
   author={Breckner, H.},
     
   title={Galerkin approximation and the strong solution of the
              Navier-Stokes equation},
   journal={Journal of Applied Mathematics and Stochastic Analysis},
   volume={13},
   date={2000},
    NUMBER = {3},
     PAGES = {239--259},


  
}

 \bib{BP00}{book}{
   author={Brze\'{z}niak, Z.},
      author={Peszat, S.}
     
   title={Infinite dimensional stochastic analysis},
    SERIES = {Verh. Afd. Natuurkd. 1. Reeks. K. Ned. Akad. Wet.},
    VOLUME = {52},
     PAGES = {85--98},
 PUBLISHER = {R. Neth. Acad. Arts Sci., Amsterdam},
      YEAR = {2000},


  
}









\bib{B70}{article}{
   author={Busse, F. H.},
   title={Bounds for turbulent shear flow},
  
   journal={Journal of Fluid Mechanics},
   volume={41},
   date={1970},
 
   pages={4219--240},
   
}


\bib{CLM01}{article}{
   author={Caraballo, T. },
   author={Liu, K. },
      author={X.R. Mao},
   title={On stabilization of partial differential equations by noise},
   journal={Nagoya Math. J.},
   volume={161(2)},
   date={2001},
   pages={155–170},
  
}




 \bib{C73}{article}{
   author={Chorin, A. J.},

   title={Numerical study of slightly visous flow},
   journal={J. Fluid Mech.},
   volume={57},
   date={1973},

   pages={785–796},
  
}

 \bib{CI08}{article}{
   author={P. Constantin},
   author={G. Iyer},
   title={A stochastic Lagrangian representation of the three-dimensional incompressible Navier–Stokes equations},
   journal={Comm. Pure Appl. Math.},
   volume={61},
   date={2008},
   pages={330–345},
  
}




 \bib{CI11}{article}{
   author={Constantin, P.},
   author={Iyer, G.},
   title={A stochastic-Lagrangian approach to the Navier–Stokes equations in domains with boundary},
   journal={Ann. Appl. Probab.},
   volume={21},
   date={2011},
   number={4},
   pages={1466-1492},
  
}


\bib{DLPRSZ18}{article}{
   author={DeCaria, V. },
   author={Layton, W.},
    author={Pakzad, A.},
     author={Rong, Y. },
      author={Sahin, N.},
       author={Zhao, H. },
   title={On the determination of the grad-div criterion},
   journal={Journal of Mathematical Analysis and Applications},
   volume={467},
   issue={2},
   date={2018},
   pages={1032--1037},
   }
   
  \bib{D09}{article}{
   author={Doering, C .R.},
 
   title={The 3D Navier-Stokes Problem},
  
   journal={Annual Review of Fluid Mechanics},
   volume={41},
   date={2009},
   pages={109-128},
   
}
  
  
  

\bib{DC92}{article}{
   author={Doering, C .R.},
   author={Constantin, P.},
   title={Energy dissipation in shear driven turbulence},
  
   journal={Physical review letters},
   volume={69},
   date={1992},
   pages={1648},
   
}


 
   \bib{DC96}{article}{
   author={Doering, C. R.},
   author={Constantin, P.},
   title={Variational bounds on energy dissipation in incompressible flows. III. Convection},
  
   journal={Phys. Rev. E},
   volume={53},
   date={1996},
   pages={5957--5981},
   
}

\bib{DF02}{article}{
   author={Doering, C. R.},
   author={Foias, C.},
   title={Energy dissipation in body-forced turbulence},
   journal={J. Fluid Mech.},
   volume={467},
   date={2002},
   pages={289--306},
 
}
\bib{DR00}{article}{
   author={Duchon, J. },
   author={Robert, R. },
   title={Inertial energy dissipation for weak solutions of incompressible Euler and Navier–Stokes equations},
   journal={Nonlinearity},
   volume={13},
   date={2000},
   pages={249--255},
 
}






\bib{E13}{book}{
   author={Evans, L.  C. },
   title={An Introduction to Stochastic Differential Equations},
   publisher={American Mathematical Society},
   pages={151},
   date={2013},
 
}



\bib{FSQD19}{article}{
   author={Fellner, K. },
   author={Sonner, S. },
   author={Tang, B. Q. },
   author={Thuan, D. D. },
   title={Stabilisation by noise on the boundary for a Chafee-Infante equation with dynamical boundary conditions},
   journal={AIMS},
   volume={24},
  date={2019},
   pages={4055-4078}
  
}

\bib{FG95}{article}{
   author={Flandoli, F.},
   author={G\polhk atarek, D.},

  TITLE = {Martingale and stationary solutions for stochastic
              Navier-Stokes equations},
   JOURNAL = {Probab. Theory Related Fields},
  FJOURNAL = {Probability Theory and Related Fields},
    VOLUME = {102},
      YEAR = {1995},
    NUMBER = {3},
     PAGES = {367--391},
  
}

\bib{FR02}{article}{
   author={Flandoli, F.},
   author={Romito, M.},

TITLE = {Partial regularity for the stochastic {N}avier-{S}tokes
              equations},
   JOURNAL = {Trans. Amer. Math. Soc.},
  FJOURNAL = {Transactions of the American Mathematical Society},
    VOLUME = {354},
      YEAR = {2002},
    NUMBER = {6},
     PAGES = {2207--2241},
  
}





\bib{FMRT01}{book}{
  title={Navier-Stokes equations and turbulence},
  author={Foias, C.},
  author={Manley, O.},
  author={Rosa, R.},
  author={Temam, R.},
  volume={83},
  year={2001},
  publisher={Cambridge University Press}
}





\bib{F95}{book}{
   author={Frisch, U.},
   title={Turbulence},
   note={The legacy of A. N. Kolmogorov},
   publisher={Cambridge University Press, Cambridge},
   date={1995},
   pages={xiv+296},

}

\bib{strong}{article}{
  title={Strong pathwise solutions of the stochastic Navier-Stokes system},
  author={Glatt-Holtz, N.},
  author={Ziane, M.},
  journal={Advances in Differential Equations},
  volume={14},
  number={5/6},
  pages={567--600},
  year={2009},
  publisher={Khayyam Publishing, Inc.}
}


\bib{H72}{article}{
   author={Howard, L. N. },
   title={Bounds on flow quantities},
  
   journal={Annual Review of Fluid Mechanics},
   volume={4},
   date={1972},
 
   pages={473--494},
   
}

\bib{JL16}{article}{
   author={ Jiang, N.},
      author={W. J. Layton},
   title={Algorithms and models for turbulence not at statistical
equilibrium},
   journal={Computers and Mathematics with Applications},
   volume={71},
   date={2016},
   pages={2352--2372},
  
}




\bib{K97}{article}{
   author={Kerswell, R. R.},
   title={Variational bounds on shear-driven turbulence and
turbulent Boussinesq convection },
   journal={Physica D},
   volume={100},
   date={1997},
   pages={355--376},
  
}

\bib{K99}{article}{
   author={Kwiecinska, A.A.},
   title={Stabilization of partial differential equations by noise},
   journal={Stoch. Proc. $\&$ Appl.},
   volume={79},
   date={1999},
   pages={179–184},
  
}

\bib{K41}{article}{
   author={Kolmogorov, A. N.},
   title={The local structure of turbulence in incompressible viscous fluid
   for very large Reynolds numbers},
   note={Translated from the Russian by V. Levin;
   Turbulence and stochastic processes: Kolmogorov's ideas 50 years on},
   journal={Proc. Roy. Soc. London Ser. A},
   volume={434},
   date={1991},
   number={1890},
   pages={9--13},
  
}




\bib{L16}{article}{
   author={Layton, W. J. },
   title={Energy dissipation in the Smagorinsky model of turbulence},
   journal={Appl. Math. Lett.},
   volume={59},
   date={2016},
   pages={56--59},

}
  


 
   \bib{L02}{article}{
   author={Layton, W. J.},
   title={Energy dissipation bounds for shear flows for a model in large
   eddy simulation},
   journal={Math. Comput. Modelling},
   volume={35},
   date={2002},
   number={13},
   pages={1445--1451},
 
}


 
   \bib{LS18}{article}{
   author={Leslie, T. M.},
    author={Shvydkoy, R. },
   TITLE = {Conditions implying energy equality for weak solutions of the
              {N}avier-{S}tokes equations},
   JOURNAL = {SIAM J. Math. Anal.},
  FJOURNAL = {SIAM Journal on Mathematical Analysis},
    VOLUME = {50},
      YEAR = {2018},
    NUMBER = {1},
     PAGES = {870--890},
 
}




\bib{Leray}{article}{
  title={Sur le mouvement d'un liquide visqueux emplissant l'espace},
  author={Leray, J.},

  journal={Acta mathematica},
  volume={63},
  pages={193--248},
  year={1934},
  publisher={Institut Mittag-Leffler}
}


\bib{lions}{book}{
  title={Quelques m{\'e}thodes de r{\'e}solution des problemes aux limites non lin{\'e}aires},
  author={Lions, J.L.},
  year={1969},
  publisher={Dunod}
}





 \bib{M94}{article}{
   author={Marchioro, C.},
   title={Remark on the energy dissipation in shear driven turbulence},
   journal={Phys. D},
   volume={74},
   date={1994},
   number={3-4},
   pages={395--398},
 
}


 \bib{MR04}{article}{
   author={Mikulevicius, R.},
      author={Rozovskii, B.L.},

   title={Stochastic Navier–Stokes equations for turbulent flows},
   journal={SIAM J. Math. Anal.},
   volume={35},
   issue={5},
   date={2004},

   pages={1250–1310},
  
}

 \bib{MR05}{article}{
   author={R. Mikulevicius,},
      author={B.L. Rozovskii},

 TITLE = {Global {$L_2$}-solutions of stochastic {N}avier-{S}tokes
              equations},
   JOURNAL = {Ann. Probab.},
  FJOURNAL = {The Annals of Probability},
    VOLUME = {33},
      YEAR = {2005},
    NUMBER = {1},
     PAGES = {137--176},
  
}



\bib{OP19}{article}{
  title={Leray's fundamental work on the Navier-Stokes equations: a modern review of" Sur le mouvement d'un liquide visqueux emplissant l'espace"},
  author={O{\.z}a{\'n}ski, W. S.},
  author={Pooley, B.C},
  journal={arXiv preprint arXiv:1708.09787},
  year={2017}
}




\bib{AP16}{article}{
   author={Pakzad, A.},
   title={Damping functions correct over-dissipation of the Smagorinsky Model},
  journal={ Mathematical Methods in the Applied Sciences},
   
volume = {40},
number = {16},
issn = {1099-1476},
pages = {5933--5945},
year = {2017},  
  
}

  
\bib{AP18}{article}{
   author={A. Pakzad,},
   title={Analysis of mesh effects on turbulence statistics},
  journal={Journal of Mathematical Analysis and Applications},
  volume = {475},
  pages = {839-860},
    year = {2019},  
  
}
\bib{AP19}{article}{
   author={Pakzad, A.},
   title={On the long time behavior of time relaxation model of fluids},
  journal={Physica D},
    volume = {408},
  year = {2020},  

  }



\bib{P00}{book}{
   author={Pope, S. B.},
   title={Turbulent flows},
   publisher={Cambridge University Press, Cambridge},
   date={2000},
   pages={xxxiv+771},
  
}


\bib{R06}{article}{
  title={Existence of martingale and stationary suitable weak solutions for a stochastic Navier--Stokes system},
  author={Romito, M.},
  journal={Stochastics: An International Journal of Probability and Stochastics Processes},
  volume={82},
  number={3},
  pages={327--337},
  year={2010},
  publisher={Taylor \& Francis}
}


\bib{bilinear}{book}{
  title={Mathematical problems of statistical hydromechanics},
  author={ Vishik, M. I.},
  author={Fursikov,  A.V. },
  volume={9},
  year={2012},
  publisher={Springer Science \& Business Media}
}
  

\bib{S77}{article}{
   author={Scheffer, V.},
   title={Hausdorff measure and the Navier–Stokes equations},
   journal={ Comm. Math. Phys.},
   volume={55},
   date={1977},
   pages={97--112},
 
}




\bib{S98}{article}{
   author={Sreenivasan, K. R.},
   title={An update on the energy dissipation rate in isotropic turbulence},
   journal={Phys. Fluids},
   volume={10},
   date={1998},
   number={2},
   pages={528--529},
 
}



\bib{V14}{article}{
   author={Vassilicos, J. C.},
   title={Dissipation in turbulent flows},
   journal={Annual Review of Fluid Mechanics},
   volume={47},
   date={2014},
  pages={95-114},
 
}






\bib{WW15}{article}{
   author={Wang, D.},
      author={Wang, H.}
 TITLE = {Global existence of martingale solutions to the three-dimensional stochastic compressible Navier-Stokes  equations},
   JOURNAL = {Differential Integral Equations},
    FJOURNAL = {Differential and Integral Equations. An International Journal for Theory \& Applications},

    VOLUME = {28},
      YEAR = {2015},
    NUMBER = {11-12},
     PAGES = {1105--1154},
  
}







\bib{W10}{article}{
   author={Wang, X.},
   title={Approximation of stationary statistical properties of dissipative dynamical systems: time discretization},
   journal={Mathematics of Computation},
   volume={79},
   date={2010},
   number={269},
   pages={259-280},
  
}

\bib{W00}{article}{
   author={X. Wang},
   title={Effect of tangential derivative in the boundary layer on time averaged energy dissipation rate},
   journal={Physica D: Nonlinear Phenomena},
   volume={144},
   date={2000},
   issue={1-2},
   pages={142--153},
  
}



\bib{W97}{article}{
   author={Wang, X.},
   title={Time-averaged energy dissipation rate for shear driven flows in
   $\bold R^n$},
   journal={Phys. D},
   volume={99},
   date={1997},
   number={4},
   pages={555--563},
  
}
  



  
  



\end{thebibliography}
\end{document}